\documentclass{article}
\usepackage{amsmath}
\usepackage{amssymb}
\usepackage{amsthm}
\usepackage{mathabx}
\usepackage{placeins}
\usepackage[misc]{ifsym}
\usepackage{epsfig} 
\usepackage{epstopdf} 
\usepackage[colorlinks=true]{hyperref}
\hypersetup{urlcolor=blue, citecolor=red}
\allowdisplaybreaks

\newtheorem{theorem}{Theorem}[section]
\newtheorem{corollary}[theorem]{Corollary}
\newtheorem{lemma}[theorem]{Lemma}

\theoremstyle{definition}

\newtheorem{remark}[theorem]{Remark}

\title{Two-term large-time asymptotic expansion of the value function for dissipative nonlinear optimal control problems} 

\author{
Veljko A\v{s}kovi\'{c}\footnote{MBDA, France (\texttt{veljkoaskovic@hotmail.com}).}
\and
Emmanuel Tr\'elat\footnote{Sorbonne Universit\'e, CNRS, Universit\'e Paris Cit\'e, Inria, Laboratoire Jacques-Louis Lions (LJLL), F-75005 Paris, France (\texttt{emmanuel.trelat@sorbonne-universite.fr}).}
\and
Hasnaa Zidani\footnote{Insa Rouen Normandie, LMI, Saint-Étienne-du-Rouvray, France (\texttt{hasnaa.zidani@insa-rouen.fr}).}
}

\date{}

\begin{document}
\maketitle

\begin{abstract}
Considering a general nonlinear dissipative finite dimensional optimal control problem in fixed time horizon $T$, we establish a two-term asymptotic expansion of the value function as $T\rightarrow+\infty$. The dominating term is $T$ times the optimal value obtained from the optimal static problem within the classical turnpike theory. The second term, of order unity, is interpreted as the sum of two values associated with  optimal stabilization problems related to the turnpike.
\end{abstract}

\section{Introduction}
The long-term asymptotic properties of the value function have been extensively explored from the perspective of partial differential equations (PDE), primarily within the framework of Hamilton-Jacobi-Bellman (HJB) equations (\cite{1}, \cite{2}, \cite{3}) and ergodic theory (\cite{4}). For instance, in \cite[Chapter VII]{1}, the authors investigate the optimal control problem with discounted Lagrange cost over an infinite time horizon. To characterize the ergodic behavior of the value function, the authors take the limit as the discount factor tends to zero and identify the limit value function as the viscosity solution of the limiting equation. In \cite{2}, assuming suitable conditions, including periodicity assumptions on the Hamiltonian, the authors characterize the large-time behavior of the solution to the first-order HJB equation as the solution of a stationary equation. The extension of these results to deterministic zero-sum differential games with two conflicting controllers has been studied, as seen in \cite{3}. It is worth noting that more general results are available, such as those presented in \cite{5}, where the authors demonstrate that, under appropriate assumptions, there exists at most one potential accumulation point (in the uniform convergence topology) of the values. This occurs when the time horizon of Cesaro means converges to infinity or the discount factor of Abel means converges to zero.

\medskip

From the classical optimal control perspective, the large-time behavior of the value function is typically inferred as a consequence of a property satisfied by a broad class of optimal control problems when the time horizon is sufficiently large. This property is known as the \emph{turnpike property}, indicating that, for certain optimal control problems with a sufficiently  large time horizon, any optimal solution tends to remain close to the optimal solution of an associated \emph{static optimization} problem for the majority of the time. This optimal static solution is referred to as the turnpike (the term originates from the concept that a turnpike represents the fastest route between two distant points, even if it is not the most direct route; see Figure~\ref{fig:tnpke_illustration}).
\begin{figure}[h]
\begin{center}
\includegraphics[width=1.0\linewidth]{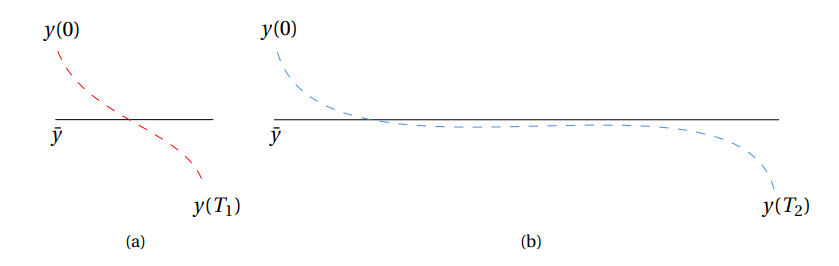}
\caption{Turnpike illustration: (a) small time case (b) large time case}
\label{fig:tnpke_illustration}	
\end{center}
\end{figure}

The turnpike phenomenon was initially observed  and investigated by economists for discrete-time optimal control problems (see \cite{6}, \cite{7}).  Various notions of turnpike properties exist, with some being stronger than others (see \cite{8}). Exponential turnpike properties have been established in \cite{9}, \cite{10}, \cite{11}, \cite{12} and \cite{13} for the optimal triple resulting of the application of Pontryagin's Maximum Principle (PMP), ensuring that the extremal solution (state, adjoint and control) remains exponentially close to an optimal solution of the corresponding static controlled problem, except at the beginning and at the end of the time interval, when the time horizon $T$ is sufficiently large. As unravelled in \cite{12} this phenomenon is closely related to hyperbolicity properties of the Hamiltonian flow. For discrete-time problems it has been shown for instance in \cite{14}, \cite{15} that the exponential turnpike property is also closely related to a strict dissipativity property. Measure-turnpike is a weaker notion of turnpike, meaning that any optimal solution, along the time frame, remains close to an optimal static solution except during a subset of times with small Lebesgue measure. It has been proved in \cite{16}, \cite{17} that measure turnpike follows from strict dissipativity or from strong duality properties.

\medskip

Based on the turnpike property, an equivalent as $T\rightarrow+\infty$ of the value function has been derived in \cite{11} or \cite{12}, (see Figure \ref{fig:tnpke}):
\begin{equation}
\frac{1}{T}.v(T,x,z) \underset{T \to +\infty}{\sim}  \bar{v} \text{ if } \bar{v} \neq 0
\label{intro1}
\end{equation}
where $v(T,x,z)$ is the optimal cost to steer the system from $x$ to $z$ in time $T$ and $\bar{v}$ is the ``steady" cost at the turnpike. 

\begin{figure}[h]
\begin{center}
\includegraphics[width=1.0\linewidth]{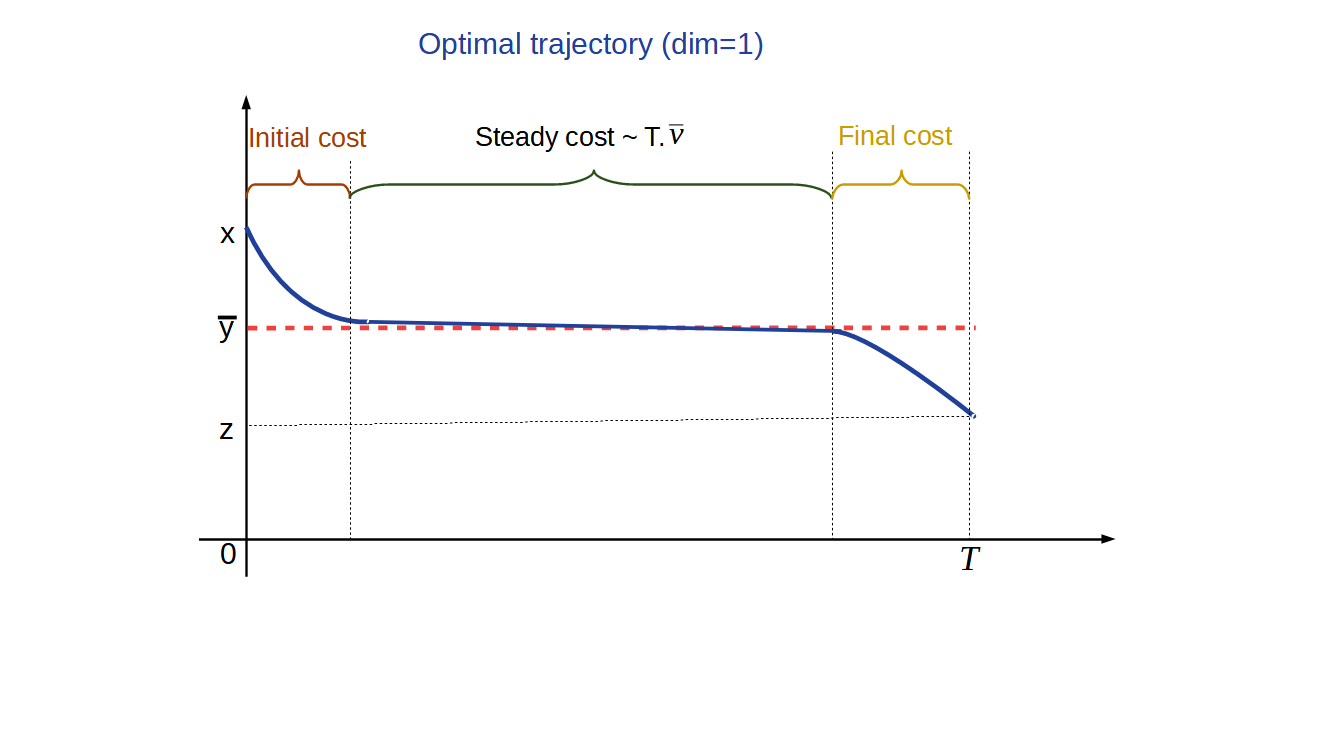}
\caption{Turnpike phenomenon}
\label{fig:tnpke}
\end{center}	
\end{figure}

Significant insights have been achieved in \cite{18} and \cite{19}, where the authors explore variations of the linear quadratic (LQ) problem. Assuming certain controllability conditions, they derive the large-time expansion of the value function at the order of one in $1/T$. In \cite{18}, the various aspects are identified, and some are interpreted within the framework of the Hamilton-Jacobi-Bellman (HJB) theory. Meanwhile, in \cite{19}, all the terms are characterized as functions of the solutions of the Riccati algebraic equation, initial/final states, and the Lagrange multiplier of the static optimization problem.

\medskip
In this manuscript, we extend the results of \cite{18,19} to the class of \emph{dissipative} nonlinear systems. The concept of (strict) dissipativity, as introduced in \cite{21}, is defined in a broad context, accompanied by related notions like the available storage function and the supply rate function. In cases where a system exhibits dissipativity with a specified supply rate function, the inquiry into identifying a suitable storage function has been extensively investigated (refer to, e.g., \cite{32}). This inquiry bears resemblance to the task of determining an appropriate Lyapunov function in the Lyapunov second method, which ensures the stability of a system. A precise mathematical definition will be provided later in this paper.
\section{Setting of the problem}
We consider the finite-dimensional optimal control problem, consisting of minimizing the cost functional
\begin{equation}
J_{T,x,z}(u)=\int_{0}^{T}f^{0}(y(t),u(t))dt
\label{costgen0}
\end{equation}
over the time interval $[0, T]$, $T>0$ being fixed under the constraints
\begin{subequations} \label{optcongen0}
\begin{align}
&\dot{y}(t)=f(y(t),u(t)), \label{optcongen01}\\
&y(0)=x, \quad y(T)=z , \label{optcongen02}
\end{align}
\end{subequations}
where $f:\mathbb{R}^{n} \times \mathbb{R}^{p} \longrightarrow \mathbb{R}^{n}$ and $f^{0}:\mathbb{R}^{n} \times \mathbb{R}^{p} \longrightarrow \mathbb{R}$ are of class $C^{1}$. The associated value function defined by
\begin{equation}
v(T,x,z):=\underset{u(\cdot) \in \mathcal{U}_{T}^{\Omega}}{\min }\hspace{0.05cm} J_{T,x,z}(u)
\label{valuegen0}
\end{equation}
where  $u \in \mathcal{U}_{T}^{\Omega}:=L^{\infty}\left([0, T], \Omega\right)$, $\Omega$ being a fixed compact subset of $\mathbb{R}^{p}$. In the sequel, we assume that $\Omega \subset \overline{B}(0,c)$ for some $c>0$.

\begin{remark} \label{remomega}
Since $\Omega$ does not depend  on $T$, the space of controls $\mathcal{U}^{\Omega}$ is independent of $T$. This is because, beyond the time horizon $T$, the controls can be trivially extended within $\Omega$ and consequently,  
the space of controls will be denoted by $\mathcal{U}^{\Omega}$.
\end{remark}

We assume that for $T$ large enough,  \eqref{costgen0}-\eqref{optcongen0}-\eqref{valuegen0} admits an  optimal solution, denoted by $(\widehat{y}_{T}(\cdot),\widehat{u}_{T}(\cdot))$. The conditions ensuring the existence of such a solution are well known (see for instance \cite{20}, \cite{37}). For instance, if the set of velocities $\{f(y,u) \hspace{0.05cm} | \hspace{0.05cm} u \in \Omega \}$ is a convex subset of $\mathbb{R}^{n}$ for any $y \in \mathbb{R}^{n}$, with mild growth at infinity and if the epigraph of $f^{0}$ is convex, then there exists at least one optimal solution. These conditions are for example satisfied if the dynamics are control-affine and if the cost functional is convex with respect to $u$

\medskip
By the Pontryagin maximum principle (\cite{20}-\cite{35}), there exist $\lambda^{0} \leqslant 0$ and an absolutely continuous mapping $\widehat{\lambda}_{T}:[0, T] \longrightarrow \mathbb{R}^{n}$ (called adjoint vector) satisfying $(\widehat{\lambda}_{T}(\cdot),\lambda^{0}) \neq (0,0)$ such that:
\begin{equation}
\begin{aligned}
&\dot{\widehat{y}}_{T}(t)=\dfrac{\partial H}{\partial \lambda}\left(\widehat{y}_{T}(t),\widehat{\lambda}_{T}(t),\lambda^{0}, \widehat{u}_{T}(t)\right)  \\
&\dot{\widehat{\lambda}}_{T}(t)=-\dfrac{\partial H}{\partial y}\left(\widehat{y}_{T}(t),\widehat{\lambda}_{T}(t),\lambda^{0}, \widehat{u}_{T}(t)\right)  \\
&\dfrac{\partial H}{\partial u}\left(\widehat{y}_{T}(t),\widehat{\lambda}_{T}(t),\lambda^{0}, \widehat{u}_{T}(t)\right)=0
\end{aligned}
\label{PMPnonlin}
\end{equation}
for almost every $t \in [0, T]$, where the Hamiltonian $H$ is defined by
\begin{equation} 
H(y, \lambda,\lambda^{0},u):=\langle \lambda , f(y,u) \rangle+\lambda^{0} f^{0}(y,u)
\label{Hamnonlin}
\end{equation}
and $\langle\ ,\rangle$ is the Euclidean scalar product in $\mathbb{R}^{n}$. We assume throughout that the abnormal case does not occur, and we set $\lambda^{0}=-1$. Moreover, we assume the existence and uniqueness of the solution (denoted by  $(\bar{y}, \bar{u})$) to the \emph{static optimization problem}
\begin{equation}
\bar{v}:=\underset{f(y,u)=0}{\min}\hspace{0.1cm}f^{0}(y,u).
\label{statnonlingen0}
\end{equation}
This is a usual nonlinear constrained optimization problem settled in $\mathbb{R}^{n} \times \mathbb{R}^{p}$. Note that the minimizer exists and is unique in the linear quadratic case.

\medskip
By the Karush Kuhn Tucker (KKT) optimality conditions, assuming that the abnormal case does not occur, there exists $\bar{\lambda} \in \mathbb{R}^{n}$ such that
\begin{equation}
\begin{aligned} 
&f(\bar{y},\bar{u})=0 \\
-&\dfrac{\partial f^{0}}{\partial y}(\bar{y},\bar{u})+\langle \bar{\lambda}, \dfrac{\partial f}{\partial y}(\bar{y},\bar{u}) \rangle=0 \\
-&\dfrac{\partial f^{0}}{\partial u}(\bar{y},\bar{u})+\langle \bar{\lambda}, \dfrac{\partial f}{\partial u}(\bar{y},\bar{u}) \rangle=0 
\end{aligned}
\label{KKTnonlin}
\end{equation}
In the Hamiltonian formalism, the first-order optimality system \eqref{KKTnonlin} is equivalent to
\begin{equation}
\begin{aligned} 
&\dfrac{\partial H}{\partial \lambda}(\bar{y},\bar{\lambda},-1, \bar{u})=0 \\
-&\dfrac{\partial H}{\partial y}(\bar{y},\bar{\lambda},-1, \bar{u})=0  \\
&\dfrac{\partial H}{\partial u}(\bar{y},\bar{\lambda},-1, \bar{u})=0
\end{aligned}
\label{KKTnonlinstat}
\end{equation}
\noindent Because of the turnpike property (see further), since it is expected that the dominating term in the asymptotic expansion of $v(\cdot)$ is $T.\bar{v}$, we "absorb" it by  subtracting it to the cost $J_{T,x,z}$ and consider the "shifted" optimal control problem defined as
\begin{subequations} \label{P0T}
\begin{align} 
&\underset{u(\cdot) \in \mathcal{U}^{\Omega} }{\min} \hspace{0.1cm} C_{T}(u,x,z):=\underset{u(\cdot) \in \mathcal{U}^{\Omega} }{\min}\hspace{0.1cm}\int_{0}^{T} w(y(t),u(t)) \, dt \label{P0T1}\\
\left(\mathcal{P}_{[0, T]}^{x,z}\right)\qquad &\dot{y}(t)=f(y(t),u(t)), \hspace{0.2cm}\forall t \in [0, T] \label{P0T2} \\
&y(0)=x, \hspace{0.1cm} y(T)=z \label{P0T3}
\end{align}
\end{subequations}
where $w(\cdot)$ is the "shifted cost" defined by\\
\begin{equation}
w(y,u):= f^{0}(y,u)-f^{0}(\bar{y},\bar{u}).
\label{shiftedcost}
\end{equation}
\noindent The readers acquainted with the notion of dissipativity in nonlinear optimal control will recognize in \eqref{shiftedcost} a classical storage function. This is where the link with dissipativity appears.

\medskip
We introduce also the shifted infinite-time optimal control problems 
\begin{equation}
\begin{aligned}
(\mathcal{P}_{\infty f}^{x}) \qquad &v_{f}(x):=\underset{u(\cdot) \in \mathcal{U}^{\Omega}}{\min}\hspace{0.1cm} \int_{0}^{+\infty} w(y(t),u(t)) \, dt \\
&\dot{y}(t)=f(y(t),u(t)), \hspace{0.2cm} y(0)=x
\end{aligned}
 \label{Pfinfnonlin}
\end{equation}
and  
\begin{equation}
\begin{aligned}
(\mathcal{P}_{\infty b}^{z}) \qquad  &v_{b}(z):=\underset{u(\cdot) \in \mathcal{U}^{\Omega}}{\min}\hspace{0.1cm} \int_{0}^{+\infty} w(y(t),u(t)) \, dt \\
&\dot{y}(t)=-f(y(t),u(t)), \hspace{0.2cm} y(0)=z
\end{aligned}
 \label{Pbinfnonlin}
\end{equation}
In our notations the index ``f'' stands for ``forward'' while the index ``b'' stands for ``backward''.
We assume the existence of optimal solutions to $(\mathcal{P}_{\infty f}^{x})$ and $(\mathcal{P}_{\infty b}^{z})$ with a finite cost.

\begin{remark} \label{fwdbwdtime}
The function $v_{f}(\cdot)$ (resp. $v_{b}(\cdot)$) is the optimal cost of the infinite-time optimal control problem consisting in steering the system from $x$ (resp. $z$) forward (resp. backward) in time and minimizing the cost functional $\int_{0}^{+\infty} w(y(t),u(t)) \, dt$. 
\end{remark}

\section{Assumptions} \hspace{10cm}
\vspace{0.2cm}

\noindent \textit{\textbf{Assumptions of global nature}}:
\vspace{0.2cm}

\noindent $(A_{1})$: (\textit{Regularity of $f$ and $f^{0}$}): We assume that $f$ and $f^{0}$ are of class $C^{1}$.
\vspace{0.2cm}

\noindent $(A_{2})$: (\textit{Existence and uniqueness of optimal solutions}): There exists $T_{0}>0$ such that for any $T \geq T_{0}$ each of the optimal control problems  $\left(\mathcal{P}_{[0, T]}^{x,z}\right)$,  $(\mathcal{P}_{\infty f}^{x})$ and  $(\mathcal{P}_{\infty b}^{z})$  admits a unique optimal solution denoted respectively by $\left(\widehat{y}_{T}(\cdot),\widehat{u}_{T}(\cdot)\right)$, $\left(\widehat{y}_{\infty f}(\cdot),\widehat{u}_{\infty f}(\cdot)\right)$ and $\left(\widehat{y}_{\infty b}(\cdot),\widehat{u}_{\infty b}(\cdot)\right)$.
\vspace{0.2cm}

\noindent $(A_{3})$: (\textit{Boundedness of the optimal trajectories}): The optimal trajectories of $\left(\mathcal{P}_{[0, T]}^{x,z}\right)$, $(\mathcal{P}_{\infty f}^{x})$ and $(\mathcal{P}_{\infty b}^{z})$ are bounded uniformly with respect to $T \geq T_{0}$:
\begin{equation}
\exists b>0, \quad \mid \quad \forall t \geqslant 0, \quad \|\widehat{y}(t)\| \leqslant b.
\label{boundy}
\end{equation}

\noindent $(A_{4})$: (\textit{Boundedness of the optimal costs}): The optimal costs of $\left(\mathcal{P}_{[0, T]}^{x,z}\right)$, $(\mathcal{P}_{\infty f}^{x})$ and $(\mathcal{P}_{\infty b}^{z})$ are bounded uniformly with respect to $T \geq T_{0}$.
\vspace{0.2cm}

\noindent $(A_{5})$: The minimizer $(\bar{y}, \bar{u})$ of \eqref{statnonlingen0} is unique and there exists a unique Lagrange multiplier $\bar{\lambda}$ satisfying \eqref{KKTnonlin}. Moreover, we assume that $\bar{u} \in \overset {\circ}{\Omega}$ (interior of $\Omega$).
\vspace{0.2cm}

In the sequel, we will also need a concept of  {\em strict dissipativity}.  
 We recall that \eqref{P0T1}-\eqref{P0T2} is \textit{dissipative} at $(\bar{y},\bar{u})$ with respect the the supply rate function $w$ if there exists a bounded function $S: \mathbb{R}^{n} \longrightarrow \mathbb{R}$, called \textit{storage function} such that for any admissible pair $(y(\cdot),u(\cdot))$ and any $T>0$:
\begin{equation}
S(y(0))+\int_{0}^{T} w (y(t),u(t)) \, dt \geqslant S(y(T)).
\label{dis_ineq}
\end{equation}
The system is \textit{strictly dissipative} if, in addition, there exists some function $\alpha(\cdot)$ of class $\mathcal{K}$ (i.e., $\alpha:[0, +\infty) \longrightarrow [0, +\infty)$ continuous, increasing,  and such that $\alpha(0)=0$) and for any $T>0$ we have:
\begin{equation}
S(y(0))+\int_{0}^{T} w (y(t),u(t)) \, dt \geqslant S(y(T))+\int_{0}^{T}\alpha\left(\left\|\begin{array}{ll}
y(t)-\bar{y} \\
u(t)-\bar{u}
\end{array} \right\|\right) \, dt.
\label{strict_dis_ineq}
\end{equation}

In this work we assume the following.
\medskip

\noindent $(A_{6})$: (\textit{Strict dissipativity property}): The family of optimal control problems \eqref{P0T1}-\eqref{P0T2} indexed by T is strictly dissipative at $(\bar{y},\bar{u})$ with respect to the supply rate function $w$ defined by \eqref{shiftedcost} with a storage function $S$.
\vspace{0.2cm}

\noindent The notion of strict dissipativity was introduced in \cite{21} and already used in \cite{14}, \cite{22}, \cite{16} and \cite{15} to derive turnpike properties. 
\vspace{0.2cm}

\noindent \textit{\textbf{Assumption of local nature}}: 
\vspace{0.2cm}

\noindent $(A_{7})$: Setting $A:=\dfrac{\partial f}{\partial y}(\bar{y},\bar{u})$, $B:=\dfrac{\partial f}{\partial u}(\bar{y},\bar{u})$, the pair $(A, B)$ satisfies the Kalman rank condition, i.e, the linearized control system at $(\bar{y},\bar{u})$ is controllable.
\vspace{0.2cm}

\noindent $(A_{8})$: (\textit{Local boundedness of the minimum time trajectories and controls near the turnpike}): There exists $r>0$ such that for any $x \in \bar{B}(\bar{y},r)$ the minimum time trajectory to $\bar{y}$ starting from $x$, denoted by $y_{x}^{\tau_{f}}(\cdot)$ and the associated control, denoted by $u^{\tau_{f}}(\cdot)$ remain in the neighbourhood of respectively $\bar{y}$ and $\bar{u}$ uniformly with respect to $x$, i.e., 
\begin{equation}
\exists r , K_{r}>0 \  \mid\ \forall x \in \bar{B}(\bar{y},r)\quad \forall t \in [0,\tau_{f}(x)] \quad  \|y_{x}^{\tau_{f}}(t)-\bar{y} \|+\|u^{\tau_{f}}(t)-\bar{u} \| \leqslant K_{r}
\label{stlc}
\end{equation}
where $\tau_{f}(\cdot)$ is the minimum time function to reach $\bar{y}$ with the dynamics $f$.

\section{Comments}

The assumptions $(A_{1})$ and $(A_{7})$ together imply that there exists $r>0$ such that for any $x \in B(\bar{y},r)$, there exists an admissible trajectory steering the control system from $x$ to $\bar{y}$ in finite time. Thus the minimum time function $\tau_{f}(\cdot)$ is well defined on $B(\bar{y},r)$ and is continuous at $\bar{y}$. This is classical result that can be found, for instance, in \cite{1}. The result remains true for the minimum time function associated to the backward-in-time dynamics $-f$, denoted by $\tau_{-f}(\cdot)$. 

\medskip
The assumption $(A_{8})$ requires the local boundedness of the minimum time trajectories and controls for any trajectory starting in the previously defined neighbourhood of the turnpike. This assumption is satisfied if the minimum time function is $C^{1}$ in the neighbourhood of $\bar{y}$. We highlight here that the regularity of the minimum time function has been widely studied in the literature: it is well known that under appropriate controllability type conditions the minimum time function has an open domain of definition and is locally Lipschitz on it, see for instance \cite{1}-\cite{23}. It is thus differentiable almost everywhere on its domain. The value function fails in general to be differentiable at points that are reached by at least two minimum time trajectories and its differentiability at a point does not guarantee continuous differentiability around this point. In \cite{24}, the authors show that, under some assumptions on the regularity and target smoothness (which excludes the singleton case), the nonemptiness of the  proximal subdifferential of the minimum time function at a point implies its continuous differentiability in a neighborhood of this point. An analogous result has been proved for the value function of the Bolza problem in \cite{25} in the case where the initial state is a prescribed point and the final state is let free. In \cite{26}, the author gives a survey of results on the regularity of the minimum time map for control-affine systems with prescribed initial and final points. Finally, for results on the set where the value function is differentiable we refer the reader to \cite{27}, \cite{28}, \cite{29}, \cite{30}, \cite{31} and references therein.

\medskip
The (strict) dissipativity property $(A_{6})$ is certainly the less intuitive assumption to check in practice. In general, when the system is dissipative, storage functions are closely related to some viscosity sub-solutions of partial differential inequalities called Hamilton-Jacobi inequalities. We refer the reader to \cite[Chapter 4]{32}  for more details on this subject. One can remark that, under suitable regularity and boundedness assumptions on the dynamics and the cost, the value function (its opposite more precisely) can be taken as a storage function, and the dissipativity inequality is then deduced from the Dynamic Programming Principle (DPP).  Note also that the infinitesimal form of the (non-strict) dissipativity inequality \eqref{dis_ineq}, with $\alpha=0$, is the Hamilton-Jacobi inequality:
\begin{equation*}
H_{1}(y, \nabla S(y),u) \leqslant -f^{0}(\bar{y}, \bar{u})
\end{equation*}
where $H_{1}$ is the maximized normal Hamiltonian (indeed, divide by $t_{1}-t_{0} > 0$ and take the limit $t_{1}-t_{0} \rightarrow 0$). The existence of $C^{1}$ solutions is therefore related to the so-called weak KAM theory (see \cite{38}). In this context, the singleton $\{(\bar{y},\bar{u})\}$ is the Aubry set and $f^{0}(\bar{y}, \bar{u})$ is the Ma\~{n}é critical value.

\begin{remark} \label{disbacdyn} The strict dissipativity inequality \eqref{strict_dis_ineq} remains true for the dynamics $-f$, provided that one switches the initial and final states. The corresponding storage function is $-S(\cdot)$.
\end{remark}

\section{Main result}

\begin{theorem} \label{thm1} Under Assumptions $(A_{1})-(A_{8})$, the value function \eqref{valuegen0} satisfies
\begin{equation}
v(T,x,z)= T.\bar{v}+v_{f}(x)+v_{b}(z)+\mathrm{o}(1) 
\label{Asyvalnonlin}
\end{equation}
as $T \to +\infty$.
\end{theorem}
\noindent In order to prove Theorem \ref{thm1}, we need some preliminary lemmas.
\subsection{Some useful lemmas}
\begin{lemma}[Barbalat's lemma] \label{Barlem}
Assume that $f:[0, +\infty) \longrightarrow \mathbb{R}$ is uniformly continuous and that $\underset{t \to +\infty}{\lim} \displaystyle{\int_{0}^{t}}f(\tau)d\tau$ exists and is finite, then $\underset{t \to +\infty}{\lim}f(t)=0$.
\end{lemma}
\begin{proof} A proof can be found for instance in \cite{36}.  By contradiction, take $\epsilon >0$ and assume that $f(t)$ does not converge to $0$ as $t \to +\infty$. In this case, there exists an increasing sequence $(t_{n})_{n \in \mathbb{N}}$ in $\mathbb{R}^{+}$ such that $|f(t_{n})| >\epsilon$. By the uniform continuity of $f$ there exists $\delta>0$ such that,  for any $n \in \mathbb{N}$, and any $t \in \mathbb{R}$
\begin{equation*}
|t-t_{n}| \leqslant \delta \implies |f(t)-f(t_{n})| \leqslant \dfrac{\epsilon}{2}.
\end{equation*}
So for any $t \in [t_{n},t_{n}+\delta]$ and any $n \in \mathbb{N}$, one has
\begin{equation*}
|f(t)|=|f(t_{n})-(f(t_{n})-f(t))| \geqslant |f(t_{n})|-|f(t_{n})-f(t)| \geqslant \dfrac{\epsilon}{2}.
\end{equation*}
Therefore,
\begin{equation*}
\left|\int_{0}^{t_{n}+\delta}f(t)dt-\int_{0}^{t_{n}}f(t)dt\right|= \left|\int_{t_{n}}^{t_{n}+\delta}f(t)dt \right|=\int_{t_{n}}^{t_{n}+\delta}|f(t)|dt \geqslant \dfrac{\delta.\epsilon}{2}.
\end{equation*}
The latter inequality contradicts the convergence of $\displaystyle{\int_{0}^{t}}f(\tau)d\tau$ as $t \to +\infty$ and the lemma follows. \end{proof}
\begin{lemma} \label{lemtnpke} The optimal trajectory $\widehat{y}_{T}(\cdot)$ of $(\mathcal{P}_{0, T}^{x,z})$ satisfies
\begin{equation}
\exists t(T) \in [0, T] \quad \mid \quad \widehat{y}_{T}\left( t(T) \right) \longrightarrow \bar{y} \text{ as }T \to +\infty.
\label{existstT}
\end{equation}
\end{lemma}
\begin{proof} For $T \geq T_{0}$, the strict dissipativity inequality applied to the optimal pair $(\widehat{y}_{T}(\cdot),\widehat{u}_{T}(\cdot))$ implies
\begin{equation}
\begin{split}
&f^{0}(\bar{y},\bar{u}) \leqslant \dfrac{1}{T}\int_{0}^{T} f^{0}(\widehat{y}_{T}(s),\widehat{u}_{T}(s)) ds+\dfrac{S(x)-S(z)}{T} \\
&\hspace{3.5cm}-\dfrac{1}{T}\int_{0}^{T}\alpha\left(\left\|\begin{array}{ll}
\widehat{y}_{T}(s)-\bar{y} \\
\widehat{u}_{T}(s)-\bar{u}
\end{array} \right\|\right) ds.
\end{split}
\label{strict_dis_ineq11}
\end{equation}
Let us prove that
\begin{equation}
\displaystyle{\dfrac{1}{T}}\int_{0}^{T}\alpha\left(\left\|\begin{array}{ll}
\widehat{y}_{T}(s)-\bar{y} \\
\widehat{u}_{T}(s)-\bar{u}
\end{array} \right\|\right) ds \longrightarrow 0 \text{ as }  T \to +\infty.
\label{convalfa}
\end{equation}
Let us assume by contradiction that this is not true: then there exists $\eta>0$ and a sequence $T_{k} \longrightarrow +\infty$ such that
\begin{equation}
\dfrac{1}{T_{k}}\int_{0}^{T_{k}}\alpha\left(\left\|\begin{array}{ll}
\widehat{y}_{T_{k}}(s)-\bar{y} \\
\widehat{u}_{T_{k}}(s)-\bar{u}
\end{array} \right\|\right) ds \geqslant \eta.
\label{int_alfa}
\end{equation}
By multiplying the inequality \eqref{strict_dis_ineq11} by $T_{k}$ one gets
\begin{equation}
\int_{0}^{T_{k}}\alpha\left(\left\|\begin{array}{ll}
\widehat{y}_{T_{k}}(s)-\bar{y} \\
\widehat{u}_{T_{k}}(s)-\bar{u}
\end{array} \right\|\right) ds \leqslant \int_{0}^{T_{k}} w(\widehat{y}_{T_{k}}(s),\widehat{u}_{T_{k}}(s)) ds+S(x)-S(z)
\label{strict_dis_ineq1bis}
\end{equation}
which implies from \eqref{int_alfa}
\begin{equation}
T_{k}.\eta \leqslant \int_{0}^{T_{k}} w(\widehat{y}_{T_{k}}(s),\widehat{u}_{T_{k}}(s)) ds+S(x)-S(z).
\label{strict_dis_ineq11bis}
\end{equation}
Assumption $(A_{4})$ leads to a contradiction in the above inequality when $k \to +\infty$.
\vspace{0.2cm}

\noindent Then we have $\displaystyle{\dfrac{1}{T}}\int_{0}^{T}\alpha\left(\left\|\begin{array}{ll}
\widehat{y}_{T}(t)-\bar{y} \\
\widehat{u}_{T}(t)-\bar{u}
\end{array} \right\|\right) ds \geqslant \displaystyle{\dfrac{1}{T}}\int_{0}^{T}\alpha\left(\left\|\widehat{y}_{T}(t)-\bar{y} \right\|\right) ds \longrightarrow 0$ as $T \to +\infty$ which implies, using the mean value theorem, that
\begin{equation}
\exists t(T) \in [0, T] \quad \mid \quad \hspace{0.2cm} \alpha\left(\left\|
\widehat{y}_{T}(t(T))-\bar{y} \right\|\right)  \longrightarrow 0 \text{ as }T \to +\infty.
\label{meanvaluethm}
\end{equation}
From the properties of  $\alpha(\cdot)$, this leads to
\begin{equation}
\widehat{y}_{T}\left( t(T) \right) \longrightarrow \bar{y} \text{ as }T \to +\infty.
\end{equation}
The lemma is proved.
\end{proof}
\begin{remark}
As noted in \cite{33}, if one makes the change of variable $s=\dfrac{t}{T}$, in \eqref{convalfa} then
\begin{equation*}
\displaystyle{\dfrac{1}{T}}\int_{0}^{T}\alpha\left(\left\|\begin{array}{ll}
\widehat{y}_{T}(t)-\bar{y} \\
\widehat{u}_{T}(t)-\bar{u}
\end{array} \right\|\right) \, dt=\displaystyle{\int_{0}^{1}}\alpha\left(\left\|\begin{array}{ll}
\widehat{y}_{T}(T.s)-\bar{y} \\
\widehat{u}_{T}(T.s)-\bar{u}
\end{array} \right\|\right) ds \longrightarrow 0 \text{ as } T \to +\infty
\end{equation*}
which implies, by the converse Lebesgue theorem (see \cite[Theorem IV.9]{34}) that there exists an increasing sequence of time horizons $(T_{k})_{k \in \mathbb{N}}$  such that $\widehat{y}_{T_{k}}(T_{k}.s) \longrightarrow \bar{y}$ and $\widehat{u}_{T_{k}}(T_{k}.s) \longrightarrow \bar{u}$ as $k \to +\infty$ for almost every $s \in [0, 1]$. The latter looks like a measure turnpike result. However we did not exploit this result in our paper.
\end{remark}

\begin{lemma} \label{lem2} The optimal trajectory $\widehat{y}_{\infty f}(\cdot)$ of $(\mathcal{P}_{\infty f}^{x})$ satisfies
\begin{equation}
\widehat{y}_{\infty f}(t) \longrightarrow \bar{y} \text{ as } t \to +\infty
\end{equation}
\end{lemma}
\begin{proof} From \eqref{strict_dis_ineq} we have
\begin{equation}
\int_{0}^{T}\alpha\left(\left\|\begin{array}{ll}
\widehat{y}_{\infty f}(t)-\bar{y} \\
\widehat{u}_{\infty f}(t)-\bar{u}
\end{array} \right\|\right)dt \leqslant \int_{0}^{T}w(\widehat{y}_{\infty f}(t),\widehat{u}_{\infty f}(t))dt + S(x)-S(\widehat{y}_{\infty f}(T))
\end{equation}
the right-hand side of the inequality being bounded uniformly with respect to $T$ one gets
\begin{equation}
\int_{0}^{T}\alpha\left(\|\widehat{y}_{\infty f}(t)-\bar{y}\|\right) \, dt \leqslant \int_{0}^{T}\alpha\left(\left\|\begin{array}{ll}
\widehat{y}_{\infty f}(t)-\bar{y} \\
\widehat{u}_{\infty f}(t)-\bar{u}
\end{array} \right\|\right)dt \underset{T \to +\infty}{=} \mathrm{O}(1)
\label{strictdisbis}
\end{equation}
which implies
\begin{equation}
\Phi(T):= \int_{0}^{T}\alpha\left(\|\widehat{y}_{\infty f}(t)-\bar{y}\|\right) \, dt \underset{T \to +\infty}{=} \mathrm{O}(1).
\label{boundPhi}
\end{equation}
First, we remark that \eqref{boundPhi} and the positivity of $\alpha(\cdot)$ imply the convergence of $\Phi(T)$ as $T \to +\infty$. 
\vspace{0.2cm}

\noindent On the other hand, from $(A_{1})$ and $(A_{3})$ we know that $f$ is of class $C^{1}$ on the compact set $\overline{B}(0,b) \times\overline{B}(0,c)$ thus bounded by a global constant, denoted by $k>0$. Therefore $\widehat{y}_{\infty f}(\cdot)$ is globally Lipschitz continuous in time $t$, and consequently $t \mapsto \|\widehat{y}_{\infty f}(t)-\bar{y} \|$ as well.
\vspace{0.2cm}

\noindent From the boundedness of  $\widehat{y}_{\infty f}(\cdot)$ (see $(A_{3})$) and the continuity of $\alpha(\cdot)$, we deduce the uniform continuity of $t \mapsto \alpha(\|\widehat{y}_{\infty f}(t)-\bar{y} \|)$ over $[0 +\infty)$. Indeed, $\alpha(\cdot)$ being continuous on a compact set, it is uniformly continuous (Heine theorem).
\vspace{0.2cm}

\noindent Applying Barbalat's lemma (see \ref{Barlem}), one has $\alpha(\|\widehat{y}_{\infty f}(t) -\bar{y}\|) \longrightarrow 0$ as $t \to +\infty$, which implies $\widehat{y}_{\infty f}(t) \rightarrow \bar{y}$ as $t \to +\infty$ and the proof is over.
\end{proof}
\begin{corollary} \label{cor1}
Lemma \ref{lem2} remains true for the optimal trajectory of $(\mathcal{P}_{\infty b}^{z})$ problem (with the backward-in-time dynamics $-f$).
\end{corollary}
\subsection{Proof of the main result}
\begin{proof} For $T \geq T_{0}$, let us consider the respective optimal solutions $(\widehat{y}_{T}(\cdot),\widehat{u}_{T}(\cdot))$, $(\widehat{y}_{\infty f}(\cdot),\widehat{u}_{\infty f}(\cdot))$ and $(\widehat{y}_{\infty b}(\cdot),\widehat{u}_{\infty b}(\cdot))$ of $\left(\mathcal{P}_{[0,T]}^{x,z}\right)$, $(\mathcal{P}_{\infty f}^{x})$ and $(\mathcal{P}_{\infty b}^{z})$.
\vspace{0.2cm}

\noindent We split the optimal cost $C_{T}(\widehat{u}_{T},x,z)$ as
\begin{align} \label{casgen}
C_{T}(\widehat{u}_{T},x,z)&=\int_{0}^{T} w(\widehat{y}_{T}(t),\widehat{u}_{T}(t)) \, dt\\
&=\underbrace{\int_{0}^{t(T)} w(\widehat{y}_{T}(t),\widehat{u}_{T}(t)) \, dt}_{C_{T}^{f}}+\underbrace{\int_{t(T)}^{T} w(\widehat{y}_{T}(t),\widehat{u}_{T}(t)) \, dt}_{C_{T}^{b}}
\end{align}
where $t(T)$ is defined by \eqref{existstT}.
\vspace{0.2cm}

\noindent We perform the proof in two steps. We first prove that
\begin{equation}  
\text{\textbf{\textcolor{blue}{Step 1:}}} \hspace{0.5cm } v_{f}(x)+v_{b}(z) \leqslant \liminf_{T \to +\infty}\hspace{0.1cm}{C_{T}(\widehat{u}_{T},x,z)} 
\label{ineqthm10}
\end{equation}
Then we prove that:
\begin{equation}  
\text{\textbf{\textcolor{blue}{Step 2:}}} \hspace{0.5cm } \limsup_{T \to +\infty}\hspace{0.1cm}{C_{T}(\widehat{u}_{T},x,z)} \leqslant v_{f}(x)+v_{b}(z) 
\label{ineqthm20}
\end{equation}
which will prove the required result.
\vspace{0.2cm}

\noindent The real number $r$ being defined in Assumption $(A_{8})$, we first remark that $(y,u) \mapsto w(y,u)$ is continuous on $\Omega_{r}:=\overline{B}(\bar{y},K_{r}) \times \overline{B}(\bar{u},K_{r})$ which is a compact set of $\mathbb{R}^{n} \times \mathbb{R}^{p}$. Consequently,
\begin{equation}
\exists M_{r}>0 \hspace{0.2cm} \mid \hspace{0.2cm} \forall (y,u) \in \Omega_{r} ,\hspace{0.2cm}  |w(y,u)|
 \leqslant M_{r}.
\label{boundw}
\end{equation}
Let $\epsilon>0$. The continuity of $\tau_{f}(\cdot)$ at $\bar{y}$ gives:
\begin{equation}
\exists \eta >0 \text{  s.t  } \|x-\bar{y} \| \leqslant \eta \Rightarrow |\tau_{f}(x)| \leqslant \dfrac{\epsilon}{2.M_{r}}.
\label{taufconty}
\end{equation}
The continuity of $\tau_{-f}(\cdot)$ at $\bar{y}$ gives:
\begin{equation}
\exists \nu >0 \text{ s.t } \|x-\bar{y} \| \leqslant \nu \Rightarrow |\tau_{-f}(x)| \leqslant \dfrac{\epsilon}{2.M_{r}}.
\label{taubconty}
\end{equation}
We set $\gamma:=\min(\eta,\nu,r)>0$, and we denote by $B:=B(\bar{y},\gamma)$.
\vspace{0.2cm}

\noindent \textcolor{blue}{$\triangleright$ \textbf{Step 1}}: From Lemma \ref{lemtnpke}, we know that:
\begin{equation}
\exists T_{1} \geqslant 0 \text{ s.t } \forall T \geqslant T_{1}, \hspace{0.2cm} \widehat{y}_{T}\left(t(T)\right) \in \overline{B}.
\label{tnpke1}
\end{equation}
We select a time horizon $T$ such that $T \geqslant \max(T_{0},T_{1})$ and construct $\widecheck{u}(\cdot)$ admissible control for the $(\mathcal{P}_{\infty f}^{x})$ problem as follows:
\begin{equation*}
	\widecheck{u}(t) :=\left\{
	\begin{array}{lll}
	\widehat{u}_{T}(t) &\text{ if } t \in \left[0, t(T)\right] \\[2mm]
	\widehat{u}_{0}(t) &\text{ if } t \in \left[t(T), t(T)+\tau_{0}\right] \\[2mm]
	\bar{u} &\text{ if } t \geqslant t(T)+\tau_{0}
	\end{array}
	\right.
\end{equation*}
where $\tau_{0}:=\tau_{f}\left( \widehat{y}_{T}\left(t(T)\right) \right)$ is the minimum  time to reach $\bar{y}$ from $\widehat{y}_{T}\left(t(T) \right)$ and $\widehat{u}_{0}(\cdot)$ the associated optimal control.
\vspace{0.2cm}

\noindent We infer from \eqref{taufconty} and \eqref{tnpke1} that
\begin{equation} 
\begin{aligned} 
v_{f}(x) &\leqslant \int_{0}^{+\infty} w\left(\widecheck{y}(t),\widecheck{u}(t)\right) \, dt \\
&\leqslant \int_{0}^{t(T)} w\left(\widehat{y}_{T}(t),\widehat{u}_{T}(t)\right) \, dt+\int_{t(T)}^{t(T)+\tau_{0}} w(\underbrace{\widehat{y}_{0}(t),\widehat{u}_{0}(t)}_{\in \Omega_{r}})dt+\int_{t(T)+\tau_{0}}^{+\infty} \underbrace{w(\bar{y},\bar{u})}_{0} dt \\
\hspace{0.01cm}&\leqslant C_{T}^{f}+\tau_{0}.M_{r} \\
&\leqslant C_{T}^{f}+\dfrac{\epsilon}{2}
\end{aligned}
\label{alfa1}
\end{equation}
For the second term $C_{T}^{b}$, we first remark that
\begin{align} \label{CTb}
C_{T}^{b}&=\int_{t(T)}^{T} w\left(\widehat{y}_{T}(t),\widehat{u}_{T}(t)\right) \, dt=\int_{0}^{T-t(T)} w\left((\widetilde{y}_{T}(t),\widetilde{u}_{T}(t)\right) \, dt
\end{align}
where $\left(\widetilde{y}_{T}(t),\widetilde{u}_{T}(t)\right):=\left(\widehat{y}_{T}(T-t),\widehat{u}_{T}(T-t)\right)$ is such that
\begin{equation}
\dot{\widetilde{y}}_{T}(t)=-f(\widetilde{y}_{T}(t),\widetilde{u}_{T}(t)) \text{ with } \widetilde{y}_{T}(0)=z
\label{dynret}
\end{equation}
Noting that the $\widetilde{y}_{T}\left(T-t(T)\right)=\widehat{y}_{T}\left(t(T)\right)$, we construct $\breve{u}(\cdot)$ the admissible control for the $(\mathcal{P}_{\infty b}^{z})$ problem as follows
\begin{equation*}
	\breve{u}(t) :=\left\{
	\begin{array}{lll}
	\widetilde{u}_{T}(t) &\text{ if } t \in \left[0, T-t(T)\right] \\[2mm]
	\widehat{u}_{1}(t) &\text{ if } t \in \left[T-t(T), T-t(T)+\tau_{1}\right] \\[2mm]
	\bar{u} &\text{ if } t \geqslant T-t(T)+\tau_{1}
	\end{array}
	\right.
\end{equation*}
where $\tau_{1}:=\tau_{-f}\left( \widetilde{y}_{T}\left(T-t(T) \right) \right)$ is the minimum  time to reach $\bar{y}$ from $\widetilde{y}_{T}\left(T-t(T) \right)$ and $\widehat{u}_{1}(\cdot)$ the associated optimal control.
\vspace{0.2cm}

\noindent We infer from \eqref{taubconty} and \eqref{tnpke1} that
\begin{equation} 
\begin{aligned} 
v_{b}(z) &\leqslant \int_{0}^{+\infty} w(\breve{y}(t),\breve{u}(t)) \, dt \\
&\leqslant \int_{0}^{T-t(T)} w(\widetilde{y}_{T}(t),\widetilde{u}_{T}(t)) \, dt \\
& \qquad\qquad\qquad +\int_{T-t(T)}^{T-t(T)+\tau_{1}} w(\underbrace{\widehat{y}_{1}(t),\widehat{u}_{1}(t)}_{\in \Omega_{r}}) \, dt+\int_{T-t(T)+\tau_{1}}^{+\infty} \underbrace{w(\bar{y},\bar{u})}_{0} dt \\
\hspace{0.01cm}&\leqslant C_{T}^{b}+\tau_{1}.M_{r} \\
&\leqslant C_{T}^{b}+\dfrac{\epsilon}{2}
\end{aligned}
\label{beta1}
\end{equation}
Combining \eqref{alfa1} and \eqref{beta1}, we obtain
\begin{equation}  
v_{f}(x) +v_{b}(z) \leqslant C_{T}^{f}+C_{T}^{b}+\epsilon=C_{T}(\widehat{u}_{T},x,z)+\epsilon
\label{ineqthm1eps}
\end{equation}
and thus
\begin{equation}  
v_{f}(x)+v_{b}(z) \leqslant \liminf_{T \to +\infty}\hspace{0.1cm}{C_{T}(\widehat{u}_{T},x,z)}
\label{ineqthm1}
\end{equation}
\vspace{0.2cm}

\noindent \textcolor{blue}{$\triangleright$  \textbf{Step 2}}: From Lemma \ref{lem2} and Corollary \ref{cor1}, we have $\widehat{y}_{\infty f}(t) \longrightarrow \bar{y}$ and $\widehat{y}_{\infty b}(t) \longrightarrow \bar{y}$ as $t \to +\infty$. Consequently:
\begin{equation}
\exists T_{2} >0 \text{ s.t } \forall T \geqslant T_{2}, \hspace{0.2cm} \left\|\widehat{y}_{\infty f}\left(\dfrac{T}{2}-1\right)-\bar{y}\right\| \leqslant \gamma
\label{yinffconv}
\end{equation}
and
\begin{equation}
\exists T_{3} >0 \text{ s.t } \forall T \geqslant T_{3}, \hspace{0.2cm} \left\|\widehat{y}_{\infty b}\left(\dfrac{T}{2}-1\right)-\bar{y}\right\| \leqslant \gamma
\label{yinfbconv}
\end{equation}
Take $T \geqslant \max\left(T_{2},T_{3}\right)$ and denote:
\begin{itemize}
\item $\tau_{3}:=\tau_{f}\left( \widehat{y}_{\infty f}\left(\dfrac{T}{2}-1\right)\right)$ and $\widehat{u}_{3}(\cdot)$ the associated optimal control;
\item $\tau_{4}:=\tau_{-f}\left( \widehat{y}_{\infty b}\left(\dfrac{T}{2}-1\right)\right)$ and $\widehat{u}_{4}(\cdot)$ the associated optimal control.
\end{itemize}
\vspace{0.2cm}
\noindent We construct the admissible control $\underline{u}_{T}(\cdot)$ for $(\mathcal{P}_{0,T})_{x,z}$ as follows (see Figure \ref{fig:traj_opt}):
\begin{equation*}
	\underline{u}_{T}(t) :=\left\{
	\begin{array}{lllll}
	\widehat{u}_{\infty f}(t) &\text{ if } t \in \left[0, \dfrac{T}{2}-1\right] \\[3mm]
	\widehat{u}_{3}(t) &\text{ if } t \in \left[\dfrac{T}{2}-1, \dfrac{T}{2}-1+\tau_{3}\right] \\[3mm]
	\bar{u} &\text{ if } t \in \left[\dfrac{T}{2}-1+\tau_{3}, \dfrac{T}{2}+1-\tau_{4}\right] \\[3mm]
	\widehat{u}_{4}(T-t) &\text{ if } t \in \left[\dfrac{T}{2}+1-\tau_{4}, \dfrac{T}{2}+1\right] \\[3mm]
	\widehat{u}_{\infty b}(T-t) &\text{ if } t \in \left[\dfrac{T}{2}+1, T\right]	
	\end{array}
	\right.
\end{equation*}
\begin{figure}
\begin{center}
\includegraphics[width=1.0\linewidth]{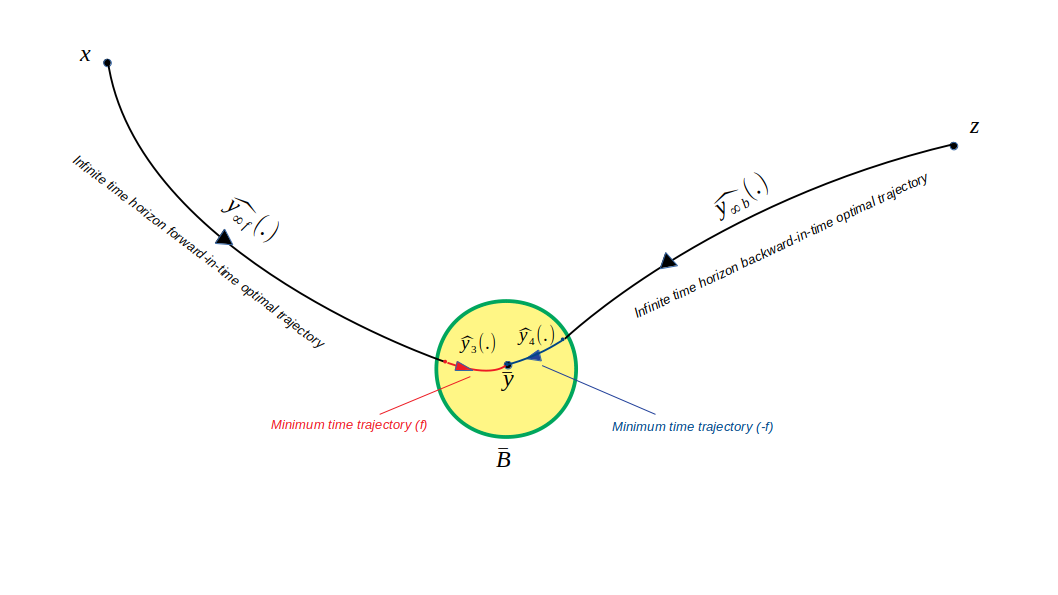}
\caption{Construction of an admissible trajectory for $\left(\mathcal{P}_{[0,T]}^{x,z}\right)$}
\label{fig:traj_opt}
\end{center}	
\end{figure}
\FloatBarrier
We have then the upper bound for the optimal cost:
\begin{equation} 
C_{T}(\widehat{u}_{T},x,z) \leqslant \int_{0}^{T} w(\underline{y}_{T}(t),\underline{u}_{T}(t)) \, dt 
\leqslant A+B+C+D+E
\label{CTmaj1}
\end{equation}
where
\vspace{0.15cm}

\noindent $A:=\displaystyle{\int_{0}^{\frac{T}{2}-1}} w\left(\widehat{y}_{\infty f}(t),\widehat{u}_{\infty f}(t)\right) \, dt$ \\
\vspace{0.15cm}

\noindent $E:=\displaystyle{\int_{\frac{T}{2}+1}^{T}} w\left(\widehat{y}_{\infty b}(T-t),\widehat{u}_{\infty b}(T-t)\right) \, dt=\displaystyle{\int_{0}^{\frac{T}{2}-1}} w(\widehat{y}_{\infty b}(t),\widehat{u}_{\infty b}(t)) \, dt$ \\
\vspace{0.15cm}

\noindent $B:=\displaystyle{\int_{\frac{T}{2}-1}^{\frac{T}{2}-1+\tau_{3}}} w(\underbrace {\widehat{y}_{3}(t),\widehat{u}_{3}(t)}_{\in \Omega_{r}}) \, dt \leqslant \tau_{3}.M_{r} \leqslant \dfrac{\epsilon}{2}$
\vspace{0.15cm}

\noindent $C:=\displaystyle{\int_{\frac{T}{2}-1+\tau_{3}}^{\frac{T}{2}+1-\tau_{4}}}\underbrace {w\left(\bar{y},\bar{u}\right)}_{0} dt=0$
\vspace{0.15cm}

\noindent $D:=\displaystyle{\int_{\frac{T}{2}+1-\tau_{4}}^{\frac{T}{2}+1}} w\left(\widehat{y}_{4}(T-t),\widehat{u}_{4}(T-t)\right) \, dt=\displaystyle{\int_{\frac{T}{2}-1}^{\frac{T}{2}-1+\tau_{4}}} w(\underbrace {\widehat{y}_{4}(t),\widehat{u}_{4}(t)}_{\in \Omega_{r}}) \, dt  \leqslant \tau_{4}.M_{r} \leqslant \dfrac{\epsilon}{2} $
\vspace{0.15cm}

\noindent Finally, we obtain
\begin{equation} 
\begin{aligned} 
C_{T}(\widehat{u}_{T},x,z) &\leqslant \displaystyle{\int_{0}^{\frac{T}{2}-1}} w\left(\widehat{y}_{\infty f}(t),\widehat{u}_{\infty f}(t)\right) \, dt+\displaystyle{\int_{0}^{\frac{T}{2}-1}} w(\widehat{y}_{\infty b}(t),\widehat{u}_{\infty b}(t)) \, dt+\epsilon
\end{aligned}
\label{CTmaj2}
\end{equation}
Noting that the integrals $A$ and $E$ converge, we take the limit superior as $T \to +\infty$ of the above inequality and we obtain:
\begin{equation} 
\limsup_{T \to +\infty}\hspace{0.1cm}  C_{T}(\widehat{u}_{T},x,z) \leqslant v_{f}(x)+v_{b}(z).
\label{CTmaj3}
\end{equation}
\vspace{0.2cm}
\textcolor{blue}{$\triangleright$ \textbf{Conclusion}}: By combining \eqref{ineqthm1} and \eqref{CTmaj3}, we obtain the required result:
\begin{equation} 
\lim_{T \to +\infty} C_{T}(\widehat{u}_{T},x,z)=v_{f}(x)+v_{b}(z).
\label{CT}
\end{equation}
\end{proof}
\begin{remark} \label{valueexplq} In the LQ case, a two-term large-time expansion of the value function has been derived in \cite{19} solely under the Kalman condition. Assumptions such as the strict dissipativity property and the existence/uniqueness of the solution of the static optimization problem are automaticaly satisfied in this case. For the proof, see \cite{19}.
\end{remark}

\section{Conclusions} Considering a general nonlinear dissipative optimal control problem in finite dimension and fixed time $T$, we have established a two-term asymptotic expansion of the value function as $T\rightarrow+\infty$. Essentially based on the strict dissipativity property, the result is a generalization of the expansion established in the variants of the linear quadratic case treated in \cite{18} and \cite{19}. We highlight here that a generalization of \cite{19} to the infinite dimensional case is currently being finalized. The infinite dimensional nonlinear case remains open.









\end{document}